\documentclass[11pt]{amsart}

\usepackage[all,cmtip]{xy}
\usepackage{amsmath, amssymb}
\usepackage{mathtools}
\usepackage[T1]{fontenc}
\usepackage[english]{babel}
\usepackage[hidelinks]{hyperref}
\usepackage[marginpar=2.5cm]{geometry}
\usepackage{xcolor}\usepackage{marginnote}
\usepackage{amsrefs}

\newtheorem{thm}{Theorem}[section]

\newtheorem*{thm*}{Theorem}
\newtheorem{lem}[thm]{Lemma}

\newtheorem{prop}[thm]{Proposition}
\newtheorem*{prop*}{Proposition}

\newtheorem{cor}[thm]{Corollary}

\theoremstyle{definition}
\newtheorem{defn}[thm]{Definition}
\newtheorem{notation}[thm]{Notation}
\newtheorem{remark}[thm]{Remark}
\newtheorem{question}[thm]{Question}

\def\e{\epsilon}

\def\al{\alpha}
\def\bb{\mathbb}

\def\de{\delta}

\def\Sg{\Sigma}

\def\bb{\mathbb}

\def\g{\gamma}

\def\G{\Gamma}
\def\cc{\mathcal}

\def\t{\mathsf{t}}

\DeclareMathOperator{\id}{id}

\DeclareMathOperator{\tr}{tr}

\DeclareMathOperator{\ch}{Ch}

\newcommand\ip[2]{\left\langle #1\, , #2 \right\rangle}

\begin{document}


\title{Remarks on the Grothendieck Norm}

\author[Sinclair]{Thomas Sinclair}
\author[Vivek]{Naveen Vivek}

\address{Mathematics Department, Purdue University, 150 N. University Street, West Lafayette, IN 47907-2067}
\email{vivek@purdue.edu}
\email{tsincla@purdue.edu}
\urladdr{http://www.math.purdue.edu/~tsincla/}


\begin{abstract}
    The goal of this short note is to point out three observations around the Grothendieck norm and semidefinite programming. The first is that the Grothendieck norm captures the difficulty of relating the off-diagonal entries of a real, symmetric matrix to a probabilistic correlation, the second is that there is an interesting ``Fourier''-type duality between the Schur and Grothendieck norms of a real matrix, and the third and last centers around the difficulty of finding an efficiently computable noncommutative analog to the Grothendieck norm.
\end{abstract}

\setcounter{tocdepth}{1}
\maketitle


\section*{Introduction}

In recent years the Schur norm $\g_2(A)$ of a matrix and its dual norm $\g_2^*(A)$, which we will refer to as the Grothendieck norm, have seen remarkable applications to combinatorial theory.
The essential reason for the interest in these norms is that they are computationally efficient to calculate via semidefinite programming, yet provide very good bounds for several computationally hard combinatorial norms thanks to Grothendieck's inequality.
We refer the reader to \cite{AlonNaor2004, ConlonZhao2017,KhotNaor,Matousek2020, Linial2007} for a non-exhaustive selection of such results, to the paper \cite{Pisier-grothendieck} for a thorough treatment of Grothendieck's Inequality for the $\g_2^*$-norm, and to \cite{Gartner2012, VandenbergheBoyd-sdp} for an introduction to semidefinite programming.

The goal of this short note is to point out three observations around the Grothendieck norm and semidefinite programming. The first is that the $\g_2^*$-norm captures the difficulty of relating the off-diagonal entries of a real, symmetric matrix to a probabilistic correlation, the second is that there is an interesting ``Fourier''-type duality between the Schur and Grothendieck norms of a real matrix, and the third and last centers around the difficulty of finding an efficiently computable noncommutative analog to the $\g_2^*$-norm. The literature here is vast, spanning the field of combinatorics, probability, functional analysis, and the quantum information theory, so the authors are hesitant to make any claims on originality. Nonetheless, we hope that collecting these thoughts together will serve to help publicize the interconnections of these fields and that the questions we pose may be a good starting point for future research.

\section{Preliminaries}

Let $A\in M_{n,k}$ be an $n\times k$ real or complex matrix. Throughout $\|A\|$ will denote the \emph{operator norm} of $A$, i.e., 
\begin{equation}
    \|A\| = \sup_{|x|_2\leq 1} |Ax|_2.
\end{equation}
It is well-known that $\|A\|$ is the largest singular value of $A$.

For each $1\leq p,q\leq \infty$ we have a matrix norm
\begin{equation}
    \|A\|_{p\to q} := \sup_{|x|_p\leq 1} |Ax|_q
\end{equation}
with $\|A\|_{2\to 2}$ coinciding with the operator norm. Note that $\|A\|_{p\to r}\leq \|A\|_{p\to q}\|A\|_{q\to r}$ and that
\begin{equation}
    \|A^*\|_{p\to q} = \|A\|_{q^*\to p^*}
\end{equation}
where $p^* = (1 - 1/p)^{-1}$ and $q^* = (1 - 1/q)^{-1}$.
It is easy to compute that 
\begin{equation}
    \begin{aligned}
        \|A\|_{1\to\infty} =& \max_{i,j} |A_{ij}|\\
         \|A\|_{\infty\to 1} =& \max_{\e_i,\eta_j\in \{\pm 1\}}\left|\sum_{i,j} A_{ij}\e_i\eta_j\right|.
    \end{aligned}
\end{equation}

\begin{defn}
    For each $1\leq p\leq \infty$ we define the \emph{Schatten $p$-norm} of a matrix $A$ by
    \begin{equation*}
        \|A\|_{S,p} := \tr(|A|^p)^{1/p}
    \end{equation*}
where $|A| = (A^* A)^{1/2}$.
\end{defn} 
Note that $\|A\|_{S,1} = \sum_{k=1}^n s_k(A)$, where $s_1(A),\dotsc, s_n(A)$ are the singular values of $A$ with multiplicity, and that $\|A\|_{S,\infty} = \max_k s_k(A) = \|A\|$, the operator norm. Further $\|A\|_{S,2} = \tr(A^* A)^{1/2}$ is just the Hilbert--Schmidt norm.

\begin{defn} 
    We define 
\begin{equation}
    \begin{aligned}
        \g_2(A) :=& \inf_{A = BC} \|B\|_{2\to\infty} \|C\|_{1\to 2}\\
                =& \inf_{A = B^* C} \|B\|_{1\to 2}\|C\|_{1\to 2}\\
                =& \inf_{A = B^* C} \|B^* B\|_{1\to \infty}^{1/2}\|C^* C\|_{1\to \infty}^{1/2}
    \end{aligned}
\end{equation}
where $B$ and $C$ are arbitrary rectangular matrices. The third line follows from the second by noting that
\[\|A\|_{1\to 2} = \max_i |A_i|_2\] where $A_i$ is the $i$-th column vector of $A$. From this we also see that the second line is equivalent to
\begin{equation}\label{eq:schur-norm-defn}
    \g_2(A) = \inf_{x_i,y_j\in\ell_2}\max_{i,j} \left\{|x_i|_2\cdot |y_j|_2 : A_{ij} = \ip{x_i}{y_j}\right\}
\end{equation}
which is often referred to as the \emph{Schur norm} of $A$. The $\g_2$-norm along with these various descriptions of it can be traced back to the seminal work of Grothendieck \cite{Grothendieck}.
\end{defn}

It is well-known and seems to have been rediscovered multiple times, see \cite[Section 2.3.2]{Lee2007}, \cite[Theorem 3.1]{Mathias1993}, \cite[Exercise 8.8(v)]{Paulsen2002}, or \cite[Section 2.6.2]{Parrilo2013}, that the last equation translates to the following semidefinite program which computes $\g_2(A)$:
\begin{equation}\label{eq:gamma2-sdp}
    \begin{aligned}
        & \text{minimize} && t\\
        & \text{subject to} && \begin{pmatrix} X & A\\ A^* & Y\end{pmatrix}\succeq 0\\
        & && X_{ii} = Y_{jj} = t.
    \end{aligned}
\end{equation}



From (\ref{eq:schur-norm-defn}) we see that the dual norm $\g_2^*$ to $\g_2$ is given by: 
\begin{equation}
    \g_2^*(A) := \sup_{|x_i|_2 = |y_j|=1} \left|\sum_{i,j} A_{ij}\ip{x_i}{y_j}\right|
\end{equation}
where $x_1,\dotsc,x_n,y_1,\dotsc,y_k\in \ell_2$, the space a square-summable sequences. Grothendieck's inequality \cite[Chapter 5]{PisierSim} shows that, in the real case, there is a constant $K_G\in (1.5, 1.8)$ so that 
\begin{equation}
    \|A\|_{\infty\to 1}\leq \g_2^*(A) \leq K_G\|A\|_{\infty\to 1}.
\end{equation}
(Note that this implies that the same holds in the complex case with constant at most $2K_G$.) For this reason, $\g_2^*$ will be referred to as the \emph{Grothendieck norm}.

As observed in, for instance, \cite{Lee2008}, \cite[Corollary 4.3]{Mathias1993}, and \cite[Section 5]{Tropp2009}, there is a semidefinite programming characterization of the Grothendieck norm $\g_2^*$; namely, $\g_2^*(A)$ is captured by the following program:
\begin{equation}\label{eq:g2dual-sdp}
    \begin{aligned}
        &\text{minimize} &&\tr(X + Y)/2\\
        &\text{subject to} && \begin{pmatrix} X & A\\ A^* & Y \end{pmatrix}\succeq 0\\
        & &&X,Y\ \text{diagonal}.
    \end{aligned}
\end{equation}

As a consequence of this, one obtains

\begin{prop}[Theorem 11, \cite{Lee2008}]\label{prop:lee11}
      For $A\in M_{n,k}$, we have that $\g^* = \g_2^*(A)$ is the minimal value so that there exist orthogonal families of vectors $\{x_1,\dots,x_n\}$ and $\{y_1,\dotsc,y_k\}$ in $\bb C^{n+k}$ with 
    \begin{equation*}
        \g^* = \sum_{i=1}^n |x_i|_2^2 = \sum_{i=1}^n |y_i|_2^2
    \end{equation*}
    so that $A_{ij} = \ip{x_i}{y_j}$.
\end{prop}

%
%

\begin{cor}\label{cor:grothendieck-decomp}
    For $A\in M_n$ we have that $\g_2^*(A)$ is the infimum of values of $|\al|_2|\beta|_2$ over all $\alpha, \beta\in \bb R^n$ such that there is a contraction $X\in M_n$ such that $A_{ij} = \alpha_i\beta_i X_{ij}$.
\end{cor}

\begin{proof}
    By Proposition \ref{prop:lee11}, we have that $\g_2^*(A)$ is the infimum of the values of 
    \[\left(\sum_{i=1}^n |x_i|^2\right)^{1/2}\left(\sum_{i=1}^n |y_i|^2\right)^{1/2}\]
    where $\{x_1,\dotsc, x_n\}$ and $\{y_1,\dotsc,y_n\}$ are orthogonal families of vectors so that $A_{ij} \equiv \ip{x_i}{y_j}$. Let $\alpha_i = |x_i|$ and $\beta_j = |y_j|$ for $i,j=1,\dotsc,n$, and let $x_i' = x_i/\alpha_i$ and $y_j' = y_j/\beta_j$, setting $x_i'$ or $y_j'$ to $0$ if the expression is undefined. We have that $A_{ij} = \alpha_i\beta_i B_{ij}$, where it is apparent by singular value decomposition that $X_{ij} := \ip{x_i'}{y_j'}$ is a contraction. \qedhere
\end{proof}

\begin{notation}
    Given $A,B\in M_n$, we define the \emph{Schur product} $A\circ B$ by 
    \begin{equation*}
        (A\circ B)_{ij} = A_{ij}B_{ij}.
    \end{equation*}
\end{notation}

\begin{prop}\label{prop:gr-norm-product}
    For $A\in M_n$, we have that
    \begin{equation}\label{eq:schur-decomp}
        \g_2^*(A) = \inf\{ \|B\|_{S,1}\|C\|_{S,\infty} : A = B\circ C\}
    \end{equation}
\end{prop}

\begin{proof}
    Let $\rho(A)$ the right hand side of equation (\ref{eq:schur-decomp}). Writing $A_{ij} = \al_i\beta_j X_{ij}$, where $X$ is a contraction, we have that $A = B\circ X$, where $B_{ij} = \al_i\beta_j$. Since $B$ is rank-one, we have that $\|B\|_{S,1} = |\al|_2|\beta|_2$. By Corollary \ref{cor:grothendieck-decomp} this shows that $\rho(A)\leq \g_2^*(A)$.
    
    Conversely, write $A = B\circ C$, where $\|C\|_{S,\infty} =1$. Let $H$ be a real Hilbert space, and let $w_1,\dotsc,w_n,z_1,\dotsc,z_n\in H$ be unit vectors. Writing $Y_{ij} = \ip{w_i}{z_j}$, we have that
    \begin{equation*}
        \left|\sum_{i,j=1}^n A_{ij}Y_{ij}\right| = \left|\sum_{i,j=1}^n B_{ij}C_{ij}Y_{ij}\right| = \left|\tr(B^* (C\circ Y))\right| \leq \|B\|_{S,1}\|C\circ Y\|_{S,\infty},
    \end{equation*}
    where the last inequality follows by duality of the $S,1$- and $S,\infty$-norms, that is,
    \[ |\tr(AB)| \leq \|A\|_{S,1}\|B\|_{S,\infty}.\] Writing $C_{ij} = \ip{x_i}{y_j}$ for orthogonal families of vectors with norms at most one, we have that the same holds for
    \[(C\circ Y)_{ij} = \ip{x_i\otimes w_i}{y_j\otimes z_j}.\]
    Since $\{x_1\otimes w_1,\dotsc, x_n\otimes w_n\}$ and $\{y_1\otimes z_1,\dotsc, y_n\otimes z_n\}$ are still orthogonal families of vectors with norms at most one, we have that $C\circ Y$ is a contraction, hence $\|C\circ Y\|_{S,\infty}\leq 1= \|C\|_{S,\infty}$. In this way $\g_2(A)\leq \rho(A)$.
\end{proof}

\begin{remark}
    For $A\in M_n$, let $\de_A: M_n\to M_n$, be linear operator given by Schur multiplication by $A$; that is, $\de_A(X) = A\circ X$. It can be easily seen that for any matrix norm $\eta$ on $M_n$ satisfying $\eta(A^* ) = \eta(A)$ that $\|\de_A\|_{\eta\to \eta} = \|\de_A\|_{\eta^*\to\eta^*}$, where $\eta^*$ is the dual norm to $\eta$. It is well-known from the work of Grothendieck that $\g_2(A) = \|\de_A\|_{S,\infty\to S,\infty}$; see, for instance, \cite[Theorem 5.1]{PisierSim}, cf.~\cite[Theorem 9]{Lee2008}. This gives an alternate way of deriving Proposition \ref{prop:gr-norm-product}.
\end{remark}

\section{The Grothendieck Norm and Correlations}

\begin{remark}
    Unless otherwise specified, in this section all matrices and vector spaces are real.
\end{remark}

Let $A\in M_n$ be a symmetric matrix with only zero entries on the main diagonal, i.e., $A_{ii} =0$, $i=1,\dotsc,n$. (We will call such a matrix \emph{hollow}.) We refer to the problem
\begin{equation}\label{eq:corr-prob}
    \begin{aligned}
        & \text{minimize} && \tr(D)\\
        & \text{s.t.} && D + A\succeq 0\\
        & && D\ \text{diagonal}
    \end{aligned}
\end{equation}
as the \emph{correlation problem}. This is equivalent to minimizing $\sum_i |x_i|^2$ where $x_1,\dotsc,x_n\in \ell_2$ are such that $\ip{x_i}{x_j} =A_{ij}$ for all $i\not=j$. To study this problem, it makes sense to find a natural norm which somehow captures the quantity we seek. 

Notice that if $A\not=0$ and $D+A\succeq 0$, then $D\not=0$ as a positive semidefinite matrix which has only zeroes on the main diagonal must have zero entries everywhere. If $D_1 + A_1\succeq 0$ and $D_2 + A_2\succeq 0$, then $(D_1+D_2) + (A_1+A_2)\succeq 0$, which shows that the output of (\ref{eq:corr-prob}) is subadditive. The only potential issue is that the output of (\ref{eq:corr-prob}) may differ for $A$ and $-A$. To address this, we introduce two norms:

\begin{defn}
    Let $A\in M_n$ be a symmetric matrix. We define
    \begin{align}
         \|A\|_C &:= \min_{D\ \textup{diagonal}} \{\tr(D) : D\succeq A\succeq -D \}\\
        \|A\|_{C'} &:= \min_{D_1,D_2\ \textup{diagonal}} \{\tr(D_1+D_2)/2 : D_1\succeq A\succeq -D_2\}.  
    \end{align}
\end{defn}

\begin{remark}
    It is easy to see that $\|A\|_{C'}\leq \|A\|_C$. Moreover, taking $D_{ii} = \max\{(D_1)_{ii}, (D_2)_{ii}\}$, it can be seen that $\|A\|_C\leq 2\|A\|_{C'}$. Taking $A$ to be the hollow matrix with all entries equal to $1$ off the main diagonal, this upper estimate of $\|A\|_{C}$ can be seen to be asymptotically sharp.
\end{remark}

\begin{lem}\label{lem:gr-sym}
Let $A$ be a symmetric matrix. We have that 
\begin{equation}
    \g_2^*(A) = \sup \sum_{i,j=1}^n A_{ij}\left(\ip{x_i}{x_j} - \ip{x_i'}{x_j'}\right)
\end{equation}
where $|x_i|_2^2 + |x_i'|_2^2 =1$ for each $i=1,\dotsc, n$.
\end{lem}

\begin{proof}
Let $\g_2^*(A) = \sum_{i,j=1}^n A_{ij}\ip{x_i}{y_j}$ with $|x_i|_2 = |y_j|_2=1$.
By symmetry of $A$ we have that 
\begin{equation*}
    \sum_{i,j=1}^n A_{ij}\ip{x_i}{y_j} = \sum_{i.j=1}^n \frac{1}{2} A_{ij}(\ip{x_i}{y_j} + \ip{x_j}{y_i})
\end{equation*}
Setting
\begin{equation*}
    z_i = \frac{1}{2}(x_i\oplus y_i + y_i\oplus x_i), \quad w_i = \frac{1}{2}(x_i\oplus y_i - y_i\oplus x_i)
\end{equation*}
it is easy to check that \[\frac{1}{2}(\ip{x_i}{y_j} + \ip{x_j}{y_i}) = \ip{z_i}{z_j} - \ip{w_i}{w_j}\] and that $|z_i|_2^2 + |w_i|_2^2 =1$. \qedhere
\end{proof}

\begin{prop}\label{prop:ACisAGR}
    Let $A\in M_n$ be a symmetric matrix. We have that $\|A\|_C = \g_2^*(A)$.
\end{prop}

\begin{proof}
    Let $J(A) := \begin{pmatrix} A & 0\\ 0 & -A\end{pmatrix}$. Using Lemma \ref{lem:gr-sym} we write the following semidefinite program which computes $\g_2^*(A)$:
    \begin{equation}
    \begin{aligned}
        &\text{maximize} && \tr(J(A)X)\\
        &\text{s.t.} && X_{ii}+ X_{i+n,i+n} = 1, &&& i=1,\dotsc,n\\
        &                  && X\succeq 0.
    \end{aligned}
\end{equation}

The dual program to this is:
\begin{equation}
    \begin{aligned}
        &\text{minimize} && \tr(D)\\
        &\text{s.t.} && \begin{pmatrix} D - A & 0\\ 0 & D+A\end{pmatrix}\succeq 0\\
        & && D\ \textup{diagonal}
    \end{aligned}
\end{equation}
which is manifestly the correlation norm of $A$. By strong duality for semidefinite programs \cite[Theorem 4.1.1]{Gartner2012} these programs compute the same value. \qedhere
    
\end{proof}

\begin{remark}
    Let $A\in M_n$ be an $n\times n$ matrix. We define
\[\widetilde A = \begin{pmatrix} 0 & A\\ A^\t & 0\end{pmatrix}\in M_{2n}(\bb R).\] Note that $\widetilde A$ is symmetric and the its eigenvalues are (with multiplicity) $\{\pm s_i : i=1,\dots,n\}$ where the $s_i$'s are the singular values of $A$. Further, it is easy to see that for any diagonal matrix $D$ that $D + \widetilde A\succeq 0$ if and only if $D - \widetilde A\succeq 0$, hence this generalizes the program (\ref{eq:g2dual-sdp}) above.
\end{remark}

We now give a characterization of $\|A\|_{C'}$ which is similar to Proposition \ref{prop:lee11} above.

\begin{prop}\label{prop:gr-orth-decomp}
    Let $A\in M_n$ be a symmetric, hollow matrix. We have that $ \|A\|_{C'}$ is the minimal value $\de^*$ so that there exist orthogonal families of vectors $\{x_1,\dots,x_n\}$ and $\{y_1,\dotsc,y_n\}$ with 
    \begin{equation*}
        \de^* = \sum_{i=1}^n |x_i|_2^2 = \sum_{i=1}^n |y_i|_2^2
    \end{equation*}
    so that $A_{ij} = \ip{x_i}{y_j}$ for all $i\not=j$.
\end{prop}

\begin{proof}
    Let $D_1, D_2$ be diagonal matrices be such that $D_1 +A\succeq 0$ and $D_2 - A\succeq 0$. Let $B = D_1+A$ and $C = D_2-A$. We have that there exist families of vectors $\{x_1,\dots,x_n\}$ and $\{y_1,\dotsc,y_n\}$ so that $B_{ij} = \ip{x_i}{x_j}$ and $C_{ij} = \ip{y_i}{y_j}$. Note that
    \begin{equation}\label{eq:energy}
        \sum_{i=1}^n |x_i|_2^2 = \tr(D_1), \quad \sum_{i=1}^n |y_i|_2^2 = \tr(D_2)
    \end{equation}
    Since $A$ is hollow.
    
    Set $z_i = \frac{1}{\sqrt 2} (x_i\oplus y_i)$ and $w_i = \frac{1}{\sqrt 2} (x_i\oplus -y_i)$ and note that for all $i\not= j$ we have that
    \begin{equation*}
        \ip{z_i}{z_j} = \ip{w_i}{w_j} = \frac{1}{2}(B_{ij} + C_{ij}) = \frac{1}{2}(A_{ij} - A_{ij}) =0;
    \end{equation*}
    hence $\{z_1,\dotsc,z_n\}$ and $\{w_1,\dotsc,w_n\}$ are orthogonal families. Further by equations (\ref{eq:energy}) it is easy to see that
    \begin{equation*}
        \sum_{i=1}^n |z_i|_2^2 = \sum_{i=1}^n |w_i|_2^2 = \frac{1}{2}(\tr(D_1) + \tr(D_2)).
    \end{equation*}
    Finally, for $i\not= j$
    \begin{equation*}
        \ip{z_i}{w_j} = \frac{1}{2}\left(\ip{x_i}{x_j} - \ip{y_i}{y_j}\right) = \frac{1}{2}\left(B_{ij} - C_{ij}\right) = A_{ij}.
    \end{equation*}
    Thus $\de^*\leq \|A\|_{C'}$.
    
    In the other direction suppose that $\{x_1,\dots,x_n\}$ and $\{y_1,\dotsc,y_n\}$ are orthogonal families of vectors so that $A_{ij} = \ip{x_i}{y_j}$ for all $i\not= j$, and not that by symmetry $A_{ij} = \ip{x_j}{y_i}$ for all $i\not= j$ as well. Similarly to Lemma \ref{lem:gr-sym} set
    \begin{equation*}
        z_i = \frac{1}{\sqrt 2}(x_i\oplus y_i + y_i\oplus x_i), \quad w_i = \frac{1}{\sqrt 2}(x_i\oplus y_i - y_i\oplus x_i).
    \end{equation*}
    We have that for all $i\not= j$ that
    \begin{equation*}
        A_{ij} = \ip{z_i}{z_j}, \quad -A_{ij} = \ip{w_i}{w_j}
    \end{equation*}
    and
    \begin{equation*}
        \frac{1}{2}\sum_{i=1}^n |z_i|_2^2 + |w_i|_2^2 = \frac{1}{2}\sum_{i=1}^n |x_i|_2^2 + |y_i|_2^2.
    \end{equation*}
    Thus $\|A\|_{C'}\leq \de^*$. \qedhere
    
\end{proof}

\section{A Fourier-Type Duality Between Grothendieck and Schur Norms}

\begin{remark}
    Unless otherwise specified, in this section all matrices and vector spaces are real.
\end{remark}

Let $A = (A_{ij})\in M_{n}$, let $\cc F(A)\in M_{2^n}$ be given by 
\begin{equation*}
    \cc F(A)_{\e\eta} = \ip{A\e}{\eta} = \sum_{i,j=1}^n A_{ij}\e_i\eta_j.
\end{equation*}
where $\e,\eta\in \{\pm 1\}^n$. Notice that $\|\cc F(A)\|_{1\to\infty} = \|A\|_{\infty\to 1}$. We have that $\cc F$ witnesses the following ``Fourier duality'' with respect to the Grothendieck and Schur norms:

\begin{prop}\label{prop:gr-schur}
 We have that $\g_2(\cc F(A))\leq \g_2^*(A) \leq \frac{\pi}{2}\,\g_2(\cc F(A))$
\end{prop}

Before proving this result, we will develop some preparatory lemmas.

\begin{defn}
Let $x_1,\dotsc,x_n\in \ell_2$ be an $n$-tuple of vectors. We define 
\begin{equation}
    \rho(x_1,\dotsc,x_n) := \max_{\e\in \{\pm 1\}^n} \biggl|\sum_{i=1}^n \e_i x_i\biggr|_2.
\end{equation}
\end{defn}

Note that if $x_1,\dotsc, x_n\in \ell_2(k)$, then for the $k\times n$-matrix $A$ whose columns are $x_1,\dotsc,x_n$ we have that
\begin{equation*}
    \|A\|_{\infty\to 2} = \rho(x_1,\dotsc,x_n).
\end{equation*}

\begin{lem}\label{prop:FA-factorize} Let $A\in M_n$ we have that 
\begin{equation*}
    \g_2(\cc F(A)) = \inf \rho(x_1,\dotsc,x_n)\rho(y_1,\dotsc,y_n)
\end{equation*}
where the infimum is taken over all vectors $x_1,\dotsc, x_n,y_1,\dotsc, y_n\in \ell^2$ so that $A_{ij} = \ip{x_i}{y_j}$. Equivalently, we have that 
\begin{equation*}\label{eq:FA-equivalent}
    \begin{split}
        \g_2(\cc F(A)) &= \inf_{A = BC} \|B\|_{2\to 1}\|C\|_{\infty\to 2}\\
        &= \inf_{A = B^\t C} \|B\|_{\infty\to 2}\|C\|_{\infty\to 2}\\
        &= \inf_{A = B^\t C} \|B^\t B\|_{\infty\to 1}^{1/2} \|C^\t C\|_{\infty\to 1}^{1/2}.
    \end{split}
\end{equation*}
\end{lem}

\begin{proof}
    Suppose $A_{ij} = \ip{x_i}{y_j}$. It follows for $\e,\eta\in \{\pm 1\}^n$ that $\cc F(A)_{\e\eta} = \ip{x(\e)}{y(\eta)}$ where 
    \begin{equation*}
        x(\e) := \sum_{i=1}^n \e_i x_i.
    \end{equation*}
    and $y(\eta)$ is defined similarly. From (\ref{eq:schur-norm-defn}) we see that
    \begin{equation*}
        \g_2(\cc F(A)) \leq \max_{\e,\eta} |x(\e)|_2 |y(\eta)|_2 = \rho(x_1,\dotsc, x_n)\rho(y_1,\dotsc,y_n).
    \end{equation*}
    
    Conversely, suppose that $\cc F(A)_{\e\eta} = \ip{x(\e)}{y(\eta)}$ for maps $x,y: \{\pm 1\}^n \to \ell_2$ with $|x(\e)|_2,|y(\eta)|_2\leq C$ for all $\e,\eta\in \{\pm 1\}^n$. Note that 
    \begin{equation*}
        A_{ij} = \frac{1}{4^{n-1}}\sum_{\e(i)=\eta(j)=1} \cc F(A)_{\e\eta};
    \end{equation*}
    hence, $A_{ij} = \ip{x_i}{y_j}$ where
    \begin{equation}\label{eq:x-i}
        x_i := \frac{1}{2^{n-1}} \sum_{\e(i)=1} x(\e), \quad y_j := \frac{1}{2^{n-1}} \sum_{\eta(j) =1} y(\eta).
    \end{equation}
    Note that by the triangle inequality we have that $|x_i|_2,|y_j|_2\leq C$ for all $i,j=1,\dotsc,n$.
    Further, it follows from the defining formulas that
    \begin{equation*}
        \ip{x_i}{y(\eta)} = \sum_{j=1}^n A_{ij}\eta_j, \quad \ip{x(\e)}{y_j} = \sum_{i=1}^n A_{ij}\e_i.
    \end{equation*}
    
    By taking orthogonal projections, we can assume without loss of generality that 
    \[{\rm span}\{x(\e) : \e\in \{\pm 1\}^n\} = {\rm span}\{y(\eta) : \eta\in \{\pm 1\}^n\}.\]
    
    Setting 
    \begin{equation*}
        \tilde x(\e) = \sum_i \e_i x_i, \quad \tilde y(\eta) = \sum_j \eta_j y_j
    \end{equation*}
    we therefore have that 
    \begin{equation*}
        \ip{x(\e) - \tilde x(\e)}{\tilde y(\eta)} = \sum_{i,j} \e_i\eta_j A_{ij} - \sum_{i,j} \e_i\eta_j A_{ij} = 0.
    \end{equation*}
    Since the $y(\eta)$'s form are a spanning set, $\tilde x(\e) = x(\e)$ and similarly $\tilde y(\eta) = y(\eta)$.
    
    We have that
    \begin{equation}\label{eq:rho-x-i}
        \begin{split}
            &\rho(x_1,\dotsc, x_n) = \max_{\e} |\tilde x_i(\e)|_2 = \max_{\e} |x(\e)|_2\leq C\\
            &\rho(y_1,\dotsc,y_n) = \max_{\eta} |\tilde y(\eta)|_2 = \max_{\eta} |y(\eta)|_2\leq C.
        \end{split}
    \end{equation} 
    This suffices to establish the equality. \qedhere
   
\end{proof}

\begin{cor}\label{cor:FA-equality-psd}
We have that $\|A\|_{\infty\to 1}\leq \g_2(\cc F(A))$ with equality if $A$ is positive semidefinite.
\end{cor}

\begin{proof}
If $A_{ij} = \ip{x_i}{y_j}$, then by the Cauchy--Schwarz inequality we have that 
\begin{equation*}
    \|A\|_{\infty\to 1}\leq \rho(x_1,\dotsc,x_n)\rho(y_1,\dotsc,y_n).
\end{equation*}

If $A$ is positive semidefinite, then $A_{ij} = \ip{x_i}{x_j}$ and \[\|A\|_{\infty\to 1} = \rho(x_1,\dotsc,x_n)^2\] by equation (\ref{eq:infty-1-psd}). Thus, $\g_2(\cc F(A))\leq \|A\|_{\infty\to 1}$ by Lemma \ref{prop:FA-factorize}.
\end{proof}   

We will require the following result, essentially due to Rietz \cite[Theorem 4]{Rietz1974}. (See also \cite[Section 4.2]{AlonNaor2004}.)

\begin{lem}\label{lem:psd-grothendieck}
Let $H$ be a Hilbert space and $x_1,\dotsc,x_n\in H$.
Then for all $y_1,\dotsc,y_n\in H $ with $\sup_i \|y_i\|\leq 1$ it holds that
\begin{equation*}
    \rho(x_1\otimes y_1,\dotsc,x_n\otimes y_n)\leq \sqrt{\frac{\pi}{2}} \rho(x_1,\dotsc,x_n)
\end{equation*}

\end{lem}

\begin{proof}
Consider the positive semidefinite matrix $A_{ij} = \ip{x_i}{x_j}$. By the Cauchy-Schwarz inequality we have for any vectors $x,y\in \bb R^n$ that 
\begin{equation*}
        \left|\ip{Ax}{y}\right| = \left|\ip{A^{1/2}x}{A^{1/2}y}\right|\leq |A^{1/2}x|_2\, |A^{1/2}y|_2
\end{equation*}
with equality if and only if $A^{1/2}x = A^{1/2}y$. It follows that 
\begin{equation}\label{eq:infty-1-psd}
    \|A\|_{\infty\to 1} = \sup_{\e\in \{\pm 1\}^n} \left|\sum_{i,j=1}^n A_{ij}\e_i\e_j\right| = \sup_{\e\in \{\pm 1\}^n} \biggl|\sum_{i=1}^n \e_ix_i\biggr|_2^2.
\end{equation}
Given $y_1,\dotsc,y_n\in H$ unit vectors we have by \cite[Theorem 4]{Rietz1974} that 
\begin{equation*}
    \left|\sum_{i=1}^n x_i \otimes y_i\right|_2^2 = \left|\sum_{i,j=1}^n A_{ij}\ip{y_i}{y_j}\right|\leq \frac{\pi}{2} \|A\|_{\infty\to 1}
\end{equation*}
which establishes the result. \qedhere  
\end{proof}

We are now prepared to prove the main result of this section.

\begin{proof}[Proof of Proposition \ref{prop:gr-schur}]
By Proposition \ref{prop:lee11}, let $x = \{x_1,\dotsc, x_n\}$ and $y = \{y_1,\dotsc, y_n\}$ be orthogonal sets in $\ell_2$ with $\sum_{i=1}^n |x_i|_2^2 = \g_2^*(A) = \sum_{j=1}^n |y_j|_2^2$ so that $A_{ij} = \ip{x_i}{y_j}$. We define maps $\tilde x, \tilde y: \{\pm 1\}^n\to H$ by
\begin{equation*}
    \tilde x(\e) = \sum_{i=1}^n \e_i x_i,\quad\ \tilde y(\eta) = \sum_{j=1}^n \eta_j y_j.
\end{equation*}
Since $\cc F(A)_{\e\eta} = \ip{\tilde x(\e)}{\tilde y(\eta)}$ for all $\e,\eta\in \{\pm 1\}^n$ we have by orthogonality that $|\tilde x(\e)|_2 |\tilde y(\eta)|_2 = \g_2^*(A)$, hence by Lemma \ref{prop:FA-factorize} that  
\[\g_2(\cc F(A)) \leq \rho(x_1,\dotsc,x_n)\rho(y_1,\dotsc,y_n) =  \g_2^*(A).\]

For the second inequality, suppose that $x,y: \{\pm 1\}^n\to H$ are such that $\cc F(A)_{\e\eta} = \ip{x(\e)}{y(\eta)}$ with \[\max_{\e} |x(\e)|_2^2 = \max_\eta |y(\eta)|_2^2 = \g_2(\cc F(A)).\]

Let $x_1, \dotsc, x_n, y_1, \dotsc, y_n$ be defined as in equation line (\ref{eq:x-i}), so that $A_{ij} = \ip{x_i}{y_j}$ and $\max_i|x_i|_2, \max_j |y_j|\leq \g_2(\cc F(A))$. Let $w_1,\dotsc,w_n,z_1,\dotsc,z_n\in \ell_2$ be unit vectors. By (\ref{eq:rho-x-i}) and Lemma \ref{lem:psd-grothendieck} we have that
\begin{equation*}
    \left|\sum_{i=1}^n x_i\otimes w_i\right|_2\cdot\left|\sum_{i=1}^n y_i\otimes z_i\right|_2\leq \frac{\pi}{2}\,\g_2(\cc F(A)). 
\end{equation*}
Thus, by the Cauchy--Schwarz inequality, we have that 
\begin{equation*}
    \left|\sum_{ij} A_{ij}\ip{w_i}{z_j}\right| = \left|\left\langle \sum_{i=1}^n x_i\otimes w_i, \sum_{j=1}^n y_j\otimes z_j\right\rangle\right | \leq \frac{\pi}{2}\,\g_2(\cc F(A)).
\end{equation*}
which establishes the inequality. \qedhere
\end{proof}

\begin{question}
    If $K = O(n)$ elements $\e_1,\dots,\e_K$ are sampled from $\{\pm 1\}^n$ independently does the Schur norm of the $K\times K$ matrix $[\ip{A\e_i}{\e_j}]$ well approximate $\|A\|_{\infty\to 1}$ with a high degree of probability?
\end{question}



\section{On a Noncommutative Grothendieck-Type Norm}

Let $\Phi: M_n\to M_n$ be a linear map. We define the \emph{conjugate} to be $\Phi^*(X) : = \Phi(X^*)^*$ and the \emph{adjoint} $\Phi^\dagger$ to be the defined by the functional equation $\tr(\Phi(X)Y) = \tr(X\Phi^\dagger(Y))$ for all $X,Y\in M_n$. 
We say that $\Phi$ is \emph{completely positive} if there are matrices $A_1,\dotsc,A_k\in M_n$ so that $\Phi(X) = \sum_{i=1}^k A_iXA_i^*$ for all $X\in M_n$.
For every $\Phi: M_n\to M_n$, we define the \emph{Choi matrix} 
\begin{equation*}
    \ch(\Phi) := \sum_{i,j} E_{ij}\otimes \Phi(E_{ij})\in M_n\otimes M_n.
\end{equation*}
Note that $\ch(\Phi^*) = \ch(\Phi)^*$ and $\ch(\Phi^\dagger) = \theta(\ch(\Phi))$ where $\theta:M_n\otimes M_n\to M_n\otimes M_n$ is the tensor flip $\theta(x\otimes y) = y\otimes x$. By Choi's theorem \cite[Theorem 3.14]{Paulsen2002} we have that $\Phi$ is completely positive if and only if $\ch(\Phi)$ is positive semidefinite. 

\begin{notation}
    For two linear operators $\Phi,\Psi: M_n\to M_n$, we write $\Phi \succeq_{cp} \Psi$ to mean that $\Phi-\Psi$ is completely positive. 
\end{notation}

An important class of linear operators on $M_n$ are the \emph{Schur multipliers}. As noted above, these are maps of the form $\de_A(X) = A\circ X$ for some $A\in M_n$. We have that $\de_A$ is completely positive if and only if $A$ is positive semidefinite \cite[Theorem 3.7]{Paulsen2002}. Further, we have that $\ch(\de_A) = \sum_{ij} A_{ij} E_{ij}\otimes E_{ij}$. It is easy to check that the map
\[\Delta_n : \sum_{ij} A_{ij} E_{ij} \to \sum_{ij} A_{ij} E_{ij}\otimes E_{ij}\]
yields a (non-unital) adjoint-preserving algebra embedding of $M_n$ into $M_n\otimes M_n$. Thus, while the Schur multipliers form a maximal abelian subalgebra of the linear operators on $M_n$, their corresponding set of Choi matrices is highly noncommutative.

The map $\Delta_n$ has a nicely behaved adjoint. To wit:

\begin{lem}\label{Delta-adjoint}
    For $\Phi: M_n\to M_n$ define $\cc E_n(\Phi)\in M_n$ by
    \begin{equation}
        \cc E_n(\Phi)_{ij} := \Phi(E_{ij})_{ij}.
    \end{equation}
    It holds that $\cc E_n(\Phi)\succeq \cc E_n(\Psi)$ if $\Phi\succeq_{cp} \Psi$. Moreover, $\cc E_n(\delta_A) = A$.
\end{lem}

\begin{proof}
    Since $\Delta_n$ induces a $\ast$-embedding of $M_n$ into $M_n\otimes M_n$, we have that $\Delta(B) = \sum_{ij} B_{ij}E_{ij}\otimes E_{ij}$ is positive semidefinite for all $B\in M_n$ positive semidefinite. 
    
    Let $\Phi$ be completely positive, i.e., $\ch(\Phi)$ is positive semidefinite. Set $A = \cc E_n(\Phi)$. We see that $0\leq \tr((\ch(\Phi)\circ \Delta_n(B))\Delta_n(J_n)) = \tr(AB)$ for all $B\in M_n$ positive semidefinite; thus, $A$ is positive semidefinite.  The second assertion is a routine computation and is left to the reader. \qedhere
\end{proof}

\begin{remark}
    It can be seen, for instance from the following lemma and (\ref{eq:g2dual-sdp}), that for a (real) matrix $A$, $\g_2^*(A)\geq \|A\|_{S,1}$. Moreover, if $A$ is diagonal, then $\g_2^*(A) = \|A\|_{S,1}$. However, it is not the case that that there is an effective bound of the form $C\|A\|_{S,1}\geq \|A\|_{\infty\to 1}$ where $C$ is independent of dimension. 

For $\Phi: M_n\to M_n$ we define 
\[\|\Phi\|_{S,\infty\to S,1} := \sup_{\|X\|\cdot\|Y\|\leq 1} \left|\tr(\Phi(X)Y)\right| = \sup_{\|X\|\leq 1} \|\Phi(X)\|_1.\]
(This is normalized so that $\|\id: M_n\to M_n\|_{S,\infty\to S,1} = n$ in comparison to $\|I_n\|_{\infty\to 1} = n$.) Our point here is to observe that in the noncommutative context an analog of the $(S,1)$-norm does indeed bound the $(S,\infty\to S,1)$-norm. We begin with a lemma which is probably well-known to experts (see \cite[Corollary 4.3]{Mathias1993}): we provide a proof for the reader's convenience.
\end{remark}

\begin{lem}\label{lem:1-norm-sdp}
   For $A\in M_n$ we have that $\|A\|_{S,1}$ is the infimum of $\tr(X+Y)/2$ so that $\begin{pmatrix} X & A\\ A^* & Y\end{pmatrix}$ is positive semidefinite. 
\end{lem}

\begin{proof}
    By singular value decomposition, there are unitary matrices $U, V$ and a diagonal matrix $D$ with non-negative entries so that $A = U^* DV$, where $\|A\|_{S,1} = \tr(D)$. If follows that \[\begin{pmatrix}
    U^* D U & A \\ A^* & V^* D V
    \end{pmatrix}\]
    is positive semi-definite, hence $\|A\|_{S,1}$ is an upper bound.
    
    On the other hand, since the trace is invariant under conjugation by orthogonal matrices we have by the same reasoning as above that the quantity is equal to the infimum of $\tr(X' + Y')/2$ so that $\begin{pmatrix} X' & D\\ D & Y'\end{pmatrix}$ is positive semidefinite. This shows that $2D_{ii}\leq X_{ii} + Y_{ii}$, which shows that $\|A\|_{S,1}$ is a lower bound. \qedhere
    
    
\end{proof}

For linear transformations $\Phi_{ij}: M_n\to M_n$, $i,j=1,2$, we write $\Phi = \begin{pmatrix} \Phi_{11} & \Phi_{12}\\ \Phi_{21} & \Phi_{22}\end{pmatrix}$ to be the linear transformation $\Phi: M_{2n}\to M_{2n}$ given by \[\Phi\begin{pmatrix} A & B\\ C & D\end{pmatrix} = \begin{pmatrix} \Phi_{11}(A) & \Phi_{12}(B)\\ \Phi_{21}(C) & \Phi_{22}(D)\end{pmatrix}.\] Notice that \[\ch(\Phi) = \begin{pmatrix} \ch(\Phi_{11}) & \ch(\Phi_{12})\\ \ch(\Phi_{21}) & \ch(\Phi_{22})\end{pmatrix}.\]

\begin{defn}
    Suppose that $\Phi: M_n\to M_n$ is linear. We define the \emph{complete Schatten 1-norm} $\|\Phi\|_{CS_1}$ to be given by the linear program
    \begin{equation}\label{cc-norm}
        \begin{aligned}
            & \text{minimize} && \frac{1}{2}\sum_{i=1}^n \tr(\Psi(E_{ii}) + \Psi'(E_{ii}))\\
            & \text{subject to} && \begin{pmatrix} \Psi & \Phi\\ \Phi^* & \Psi'
            \end{pmatrix}\succeq_{cp} 0.
        \end{aligned}
    \end{equation}
\end{defn}

We observe that the program (\ref{cc-norm}) is equivalent to the following semidefinite program.
    \begin{equation}\label{cc-norm-sdp}
        \begin{aligned}
            & \text{minimize} && \frac{1}{2}(\tr(\ch(\Psi))+\tr(\ch(\Psi')))\\
            & \text{subject to} && \begin{pmatrix} \ch(\Psi) & \ch(\Phi)\\ \ch(\Phi)^* & \ch(\Psi')
            \end{pmatrix}\succeq 0.
        \end{aligned}
    \end{equation}
    
Hence by Lemma \ref{lem:1-norm-sdp} we have the following.

\begin{cor} For $\Phi: M_n\to M_n$ linear we have that $\|\Phi\|_{CS,1} = \|\ch(\Phi)\|_{S,1}$.
\end{cor}

\begin{prop}\label{prop:cc-schatten1}
    For all $\Phi: M_n\to M_n$ linear, we have that $\|\Phi\|_{CS_1} \geq \|\cc E_n(\Phi)\|_{S,1}$. Moreover, we have that $\|\de_A\|_{CS_1} = \|A\|_{S,1}$ for all $A\in M_n$.
\end{prop}

\begin{proof}
    Let $\Psi,\Psi'$ satisfy the constraints of (\ref{cc-norm}) with respect to $\Phi$. Setting $A = \cc E_n(\Phi)$, $X = \cc E_n(\Psi)$ and $Y=\cc E_n(\Psi')$, we have by Lemma \ref{Delta-adjoint} that 
    \[\begin{pmatrix} X & A\\ A^* & Y\end{pmatrix} \succeq 0\]
    Since $\ch(\Psi)$ is positive semidefinite, we have that $\Psi(E_{ii})_{jj}\geq 0$ for all $i,j$; hence,
    \[\tr(X) = \sum_{i=1}^n \Psi(E_{ii})_{ii}\leq \sum_{i,j=1}^n \Psi(E_{ii})_{jj} = \tr(\ch(\Psi)).
    \]
    Similarly, $\tr(Y)\leq \tr(\ch(\Psi'))$. Minimizing and applying Lemma \ref{lem:1-norm-sdp}, we obtain that $\|\Phi\|_{CS,1}\geq \|A\|_{S,1}$. 
    
    For the second assertion we need only check that $\|\de_A\|_{CS,1} \leq \|A\|_{S,1}$. Let $X,Y\in M_n$ be such that $\begin{pmatrix} X & A\\ A^* & Y\end{pmatrix}\succeq 0$. It is easy to see that
    \[\begin{pmatrix} \ch(\de_X) & \ch(\de_A)\\ \ch(\de_A^*) & \ch(\de_Y) \end{pmatrix}\succeq 0,\]
    and that $\tr(\ch(\de_X)) = \tr(\Delta_n(X)) = \tr(X)$. Thus, minimizing and applying Lemma \ref{lem:1-norm-sdp} again we obtain that $\|\de_A\|_{CS,1}\leq \|A\|_{S,1}$.
\end{proof}

\begin{prop}
 We have that $\|\Phi\|_{S,\infty\to S,1}\leq \|\Phi\|_{CS,1}$.
\end{prop}

\begin{proof}
    Choose $\Psi: M_n\to M_n$ so that $\begin{pmatrix} \Psi & \Phi\\ \Phi^* & \Psi' \end{pmatrix}$ is completely positive. Let $X,Y\in M_n$. By polar decomposition we see that
    \[\begin{pmatrix} |X| & X\\ X^* & |X^*| \end{pmatrix}\ \text{and}\ \begin{pmatrix} |Y| & Y\\ Y^* & |Y^*| \end{pmatrix}\]
    are positive semidefinite; hence, so is
    \[\begin{pmatrix} \Psi(|X|)\otimes |Y| & \Phi(X)\otimes Y\\ \Phi(X)^*\otimes Y^* & \Psi'(|X^*|)\otimes |Y^*|\end{pmatrix}.\]
    Let $\Sg_n := \sum_{ij} E_{ij}\otimes E_{ij}\in M_n\otimes M_n$, and note that 
    \begin{equation*}
        \tr((X\otimes Y)\Sg_n) = \tr(XY).
    \end{equation*}
    Applying this we conclude that 
    \begin{equation}
        2\tr(\Phi(X)Y)\leq \tr(\Psi(|X|)|Y|) + \tr(\Psi'(|X^*|)|Y^*|).
    \end{equation}
    We can assume that $\|X\|, \|Y\|\leq 1$ so that $|X|, |X^*|\preceq I_n$. Thus,
    \begin{equation}
        \begin{aligned}
            \tr(\Psi(|X|)|Y|) + \tr(\Psi(|X^*|)|Y^*|) &\leq \tr(\Psi(I_n)|Y|) + \tr(\Psi'(I_n)|Y^*|)\\ &\leq \|\Psi(I_n)\|_{S,1} + \|\Psi'(I_n)\|_{S,1}\\
            &= \sum_{i=1}^n \tr(\Psi(E_{ii})) + \tr(\Psi'(E_{ii}))\\
            &= \tr(\ch(\Psi)) + \tr(\ch(\Psi')).
        \end{aligned}
    \end{equation}
    Altogether this shows that $\|\Phi\|_{S,\infty\to S,1}\leq \|\Phi\|_{CS,1}$. \qedhere
\end{proof}

\begin{cor}
    We have that $\|\de_A\|_{S,\infty\to S,1}\leq \|A\|_{S,1}$.
\end{cor}

\begin{remark} \label{rmk:cs1-bad}
    One could hope that there is a constant $K$ so that for all $n\in\bb N$ and all $\Phi: M_n\to M_n$ we have $\|\Phi\|_{CS,1} \leq K\|\Phi\|_{S,\infty\to S,1}$. In this way $\|\Phi\|_{CS,1}$ would be an effectively computable norm corresponding to the noncommutative Grothendieck inequality due to Pisier \cite{Pisier1978} and Haagerup \cite{Haagerup1985grothendieck}. Unfortunately, this is not the case.
    
    For $A\in M_n$ consider $\Phi_A: M_n\to M_n$ defined by $\Phi_A(E_{ij}) = \de_{ij}\cdot (\sum_j A_{ij} E_{jj})$. It is not hard to check that 
    \[\|\Phi_A\|_{CS,1} = \sum_{i,j} |A_{ij}|\]
    while 
    \[\|\Phi_A\|_{S,\infty\to S,1} = \|A\|_{\infty\to 1}.\]
\end{remark}

\begin{remark}
    For $\Phi: M_n\to M_n$, we define \[\G^*(\Phi) := \sup_k\|\Phi\otimes\id_{M_k}\|_{S,\infty\to S,1}.\] For the case of $\Phi_A$ is in the previous remark, we have that
    \begin{equation}
        \G^*(\Phi_A) = \sup_{\|B_i\|,\|C_j\|\leq 1} \left|\sum_{i,j} A_{ij} \tr(B_iC_j)\right|.
    \end{equation}
    There is a construction using Clifford algebras which for unit vectors $\xi_1,\dotsc,\xi_n,\eta_1,\dotsc,\eta_n\in \bb R^n$ produces contractive matrices $B_1,\dotsc,B_n,C_1,\dotsc,C_n\in M_{2^n}(\bb R)$ so that $\ip{\xi_i}{\eta_j} = \tr(B_iC_j)$: see, for instance, \cite[Section 11.1]{ABMB-book} or \cite{Tsirelson1985}. Thus,
    \begin{prop}
        For $A\in M_n(\bb R)$ and the associated map $\Phi_A: M_n(\bb R)\to M_n(\bb R)$ it holds that \[\G^*(\Phi_A) = \g_2^*(A).\]
    \end{prop}
    
    Work of Watrous \cite{Watrous2009, Watrous2013} shows that for $\Phi: M_n\to M_n$ the cb-norm, which is \[\|\Phi\|_{cb} :=  \sup_k\|\Phi\otimes \id_{M_k}\|_{S,\infty\to S,\infty} =\|\Phi\otimes \id_{M_n}\|_{S,\infty\to S,\infty},\] and the ``diamond'' norm, which is \[\|\Phi\|_{\diamond} := \sup_k\|\Phi\otimes\id_{M_k}\|_{S,1\to S,1} = \|\Phi\otimes\id_{M_n}\|_{S,1\to S,1},\] are both \emph{efficiently computable} by semidefinite programs up to arbitrary precision in polynomial time as described in \cite[Theorem 2.4]{Watrous2009}. This relies on the norms stabilizing after tensoring at rank $n$ (more generally, this requires stabilization at rank $p(n)$ for some polynomial). Since we have only shown that $\G^*(\Phi)$ stabilizes (if it does) above rank $2^n$ in only the real case, whether $\G^*$ can be computed efficiently by a semidefinite program remains open.
\end{remark}

\begin{question}
   Is there a uniform constant $K$ so that for all $n$ and all $\Phi: M_n\to M_n$ linear we have that
   \[\G^*(\Phi)\leq K\|\Phi\|_{S,\infty\to S,1}?\]
\end{question}

\begin{remark}
    Let $\cc B_n$ be the operator norm unit ball in $M_n$. Let $\cc C_n\in M_n\otimes M_n$ be the convex hull of the the set $\{x\otimes y : x,y\in \cc B_n\}$. It is clear that $\cc C_n\subset \cc B_{n^2}$. The results in this section can be proved in the following way. We observe that $\|\Phi\|_{CS,1} = \sup_{X\in \cc C_n} |\ip{\ch(\Phi)}{X}|$ and $\|\Phi\|_{CS,1} = \sup_{Y\in \cc B_{n^2}} |\ip{\ch(\Phi)}{Y}|$. Therefore, while $\cc B_{n^2}$ is the most natural semidefinite relaxation of the set $\cc C_n$, the noncommutative Grothendieck inequality suggests that it is seemingly far from an optimal one. In fact, to avoid the pitfall given in Remark \ref{rmk:cs1-bad} any ``good'' relaxation needs to have low-complexity intersection with the real  diagonal matrices in $M_{n^2}$ in the sense of having relatively few extreme points. This eliminates considering sets such as the set of all contractions in $M_n\otimes M_n$ whose partial transpose is again a contraction.

\end{remark}

For a family of convex sets $\cc D_n$ with $\cc C_n\subset \cc D_n$ and $\Phi:M_n\to M_n$ linear, we define
\[\|\Phi\|_{C,\cc D} := \sup_{X\in \cc D_n} |\ip{\ch(\Phi)}{X}|.\]

Another natural relaxation of $\cc C_n$ is to consider the convex set $\cc H_n$ of sums of the form
\[\sum_{ij} x_{ij}\, A_i\otimes B_j\]
where $x = (x_{ij})$ is a positive semidefinite matrix of trace at most one and $A_i,B_j\in \cc B_n$. Let $\cc H_n^\circ$ be the polar of $\cc H_n$. Note that since $\cc B_n^\circ = \cc B_n$ we have that $\cc C_n\subset \cc H_n^\circ$ as well.

\begin{question}
   Is there an intrinsic characterization of $\cc H_n^\circ$? Is $\cc H_n^\circ\subset \cc B_{n^2}$?
\end{question}

\begin{question}
   Are $\|\Phi\|_{C,\cc H}$ and $\|\Phi\|_{C,\cc H^\circ}$ computable by semidefinite programs?
\end{question}

\section*{Acknowledgments}

The authors were supported by NSF grants DMS-1600857 and DMS-2055155.

\bibliographystyle{amsalpha}
\bibliography{bibliography}

\begin{bibdiv}
\begin{biblist}

\bib{AlonNaor2004}{inproceedings}{
      author={Alon, Noga},
      author={Naor, Assaf},
       title={Approximating the cut-norm via {G}rothendieck's inequality},
        date={2004},
   booktitle={Proceedings of the 36th {A}nnual {ACM} {S}ymposium on {T}heory of
  {C}omputing},
   publisher={ACM, New York},
       pages={72\ndash 80},
         url={https://doi-org.ezproxy.lib.purdue.edu/10.1145/1007352.1007371},
      review={\MR{2121587}},
}

\bib{ABMB-book}{book}{
      author={Aubrun, Guillaume},
      author={Szarek, Stanis\l aw~J.},
       title={Alice and {B}ob meet {B}anach},
      series={Mathematical Surveys and Monographs},
   publisher={American Mathematical Society, Providence, RI},
        date={2017},
      volume={223},
        ISBN={978-1-4704-3468-7},
         url={https://doi-org.ezproxy.lib.purdue.edu/10.1090/surv/223},
        note={The interface of asymptotic geometric analysis and quantum
  information theory},
      review={\MR{3699754}},
}

\bib{ConlonZhao2017}{article}{
      author={Conlon, David},
      author={Zhao, Yufei},
       title={Quasirandom {C}ayley graphs},
        date={2017},
     journal={Discrete Anal.},
       pages={Paper No. 6, 14},
         url={https://doi-org.ezproxy.lib.purdue.edu/10.19086/da.1294},
      review={\MR{3631610}},
}

\bib{Gartner2012}{book}{
      author={G\"{a}rtner, Bernd},
      author={Matou\v{s}ek, Ji\v{r}\'{\i}},
       title={Approximation algorithms and semidefinite programming},
   publisher={Springer, Heidelberg},
        date={2012},
        ISBN={978-3-642-22014-2; 978-3-642-22015-9},
  url={https://doi-org.ezproxy.lib.purdue.edu/10.1007/978-3-642-22015-9},
      review={\MR{3015090}},
}

\bib{Grothendieck}{article}{
      author={Grothendieck, A.},
       title={R\'{e}sum\'{e} de la th\'{e}orie m\'{e}trique des produits
  tensoriels topologiques},
        date={1953},
        ISSN={0373-1375},
     journal={Bol. Soc. Mat. S\~{a}o Paulo},
      volume={8},
       pages={1\ndash 79},
      review={\MR{94682}},
}

\bib{Haagerup1985grothendieck}{article}{
      author={Haagerup, Uffe},
       title={The {G}rothendieck inequality for bilinear forms on
  {$C^\ast$}-algebras},
        date={1985},
        ISSN={0001-8708},
     journal={Adv. in Math.},
      volume={56},
      number={2},
       pages={93\ndash 116},
  url={https://doi-org.ezproxy.lib.purdue.edu/10.1016/0001-8708(85)90026-X},
      review={\MR{788936}},
}

\bib{KhotNaor}{article}{
      author={Khot, Subhash},
      author={Naor, Assaf},
       title={Grothendieck-type inequalities in combinatorial optimization},
        date={2012},
        ISSN={0010-3640},
     journal={Comm. Pure Appl. Math.},
      volume={65},
      number={7},
       pages={992\ndash 1035},
         url={https://doi-org.ezproxy.lib.purdue.edu/10.1002/cpa.21398},
      review={\MR{2922372}},
}

\bib{Lee2007}{article}{
      author={Lee, Troy},
      author={Shraibman, Adi},
       title={Lower bounds in communication complexity},
        date={2007},
        ISSN={1551-305X},
     journal={Found. Trends Theor. Comput. Sci.},
      volume={3},
      number={4},
       pages={front matter, 263\ndash 399 (2009)},
         url={https://doi-org.ezproxy.lib.purdue.edu/10.1561/0400000040},
      review={\MR{2558900}},
}

\bib{Lee2008}{incollection}{
      author={Lee, Troy},
      author={Shraibman, Adi},
      author={\v{S}palek, Robert},
       title={A direct product theorem for discrepancy},
        date={2008},
   booktitle={Twenty-{T}hird {A}nnual {IEEE} {C}onference on {C}omputational
  {C}omplexity},
   publisher={IEEE Computer Soc., Los Alamitos, CA},
       pages={71\ndash 80},
         url={https://doi-org.ezproxy.lib.purdue.edu/10.1109/CCC.2008.25},
      review={\MR{2513489}},
}

\bib{Linial2007}{article}{
      author={Linial, Nati},
      author={Mendelson, Shahar},
      author={Schechtman, Gideon},
      author={Shraibman, Adi},
       title={Complexity measures of sign matrices},
        date={2007},
        ISSN={0209-9683},
     journal={Combinatorica},
      volume={27},
      number={4},
       pages={439\ndash 463},
  url={https://doi-org.ezproxy.lib.purdue.edu/10.1007/s00493-007-2160-5},
      review={\MR{2359826}},
}

\bib{Mathias1993}{article}{
      author={Mathias, Roy},
       title={Matrix completions, norms and {H}adamard products},
        date={1993},
        ISSN={0002-9939},
     journal={Proc. Amer. Math. Soc.},
      volume={117},
      number={4},
       pages={905\ndash 918},
         url={https://doi-org.ezproxy.lib.purdue.edu/10.2307/2159515},
      review={\MR{1116267}},
}

\bib{Matousek2020}{article}{
      author={Matou\v{s}ek, Ji\v{r}\'{\i}},
      author={Nikolov, Aleksandar},
      author={Talwar, Kunal},
       title={Factorization norms and hereditary discrepancy},
        date={2020},
        ISSN={1073-7928},
     journal={Int. Math. Res. Not. IMRN},
      number={3},
       pages={751\ndash 780},
         url={https://doi-org.ezproxy.lib.purdue.edu/10.1093/imrn/rny033},
      review={\MR{4073932}},
}

\bib{Parrilo2013}{incollection}{
      author={Parrilo, Pablo~A.},
       title={Semidefinite optimization},
        date={2013},
   booktitle={Semidefinite optimization and convex algebraic geometry},
      series={MOS-SIAM Ser. Optim.},
      volume={13},
   publisher={SIAM, Philadelphia, PA},
       pages={3\ndash 46},
      review={\MR{3050241}},
}

\bib{Paulsen2002}{book}{
      author={Paulsen, Vern},
       title={Completely bounded maps and operator algebras},
      series={Cambridge Studies in Advanced Mathematics},
   publisher={Cambridge University Press, Cambridge},
        date={2002},
      volume={78},
        ISBN={0-521-81669-6},
      review={\MR{1976867}},
}

\bib{Pisier1978}{article}{
      author={Pisier, Gilles},
       title={Grothendieck's theorem for noncommutative {$C^{\ast} $}-algebras,
  with an appendix on {G}rothendieck's constants},
        date={1978},
        ISSN={0022-1236},
     journal={J. Functional Analysis},
      volume={29},
      number={3},
       pages={397\ndash 415},
  url={https://doi-org.ezproxy.lib.purdue.edu/10.1016/0022-1236(78)90038-1},
      review={\MR{512252}},
}

\bib{PisierSim}{book}{
      author={Pisier, Gilles},
       title={Similarity problems and completely bounded maps},
     edition={expanded},
      series={Lecture Notes in Mathematics},
   publisher={Springer-Verlag, Berlin},
        date={2001},
      volume={1618},
        ISBN={3-540-41524-6},
         url={https://doi-org.ezproxy.lib.purdue.edu/10.1007/b55674},
        note={Includes the solution to ``The Halmos problem''},
      review={\MR{1818047}},
}

\bib{Pisier-grothendieck}{article}{
      author={Pisier, Gilles},
       title={Grothendieck's theorem, past and present},
        date={2012},
        ISSN={0273-0979},
     journal={Bull. Amer. Math. Soc. (N.S.)},
      volume={49},
      number={2},
       pages={237\ndash 323},
  url={https://doi-org.ezproxy.lib.purdue.edu/10.1090/S0273-0979-2011-01348-9},
      review={\MR{2888168}},
}

\bib{Rietz1974}{article}{
      author={Rietz, Ronald~E.},
       title={A proof of the {G}rothendieck inequality},
        date={1974},
        ISSN={0021-2172},
     journal={Israel J. Math.},
      volume={19},
       pages={271\ndash 276},
         url={https://doi-org.ezproxy.lib.purdue.edu/10.1007/BF02757725},
      review={\MR{367628}},
}

\bib{Tropp2009}{inproceedings}{
      author={Tropp, Joel~A.},
       title={Column subset selection, matrix factorization, and eigenvalue
  optimization},
        date={2009},
   booktitle={Proceedings of the {T}wentieth {A}nnual {ACM}-{SIAM} {S}ymposium
  on {D}iscrete {A}lgorithms},
   publisher={SIAM, Philadelphia, PA},
       pages={978\ndash 986},
      review={\MR{2807539}},
}

\bib{Tsirelson1985}{incollection}{
      author={Tsirelson, B.~S.},
       title={Quantum analogues of {B}ell's inequalities. {T}he case of two
  spatially divided domains},
        date={1985},
      volume={142},
       pages={174\ndash 194, 200},
        note={Problems of the theory of probability distributions, IX},
      review={\MR{788202}},
}

\bib{VandenbergheBoyd-sdp}{article}{
      author={Vandenberghe, Lieven},
      author={Boyd, Stephen},
       title={Semidefinite programming},
        date={1996},
        ISSN={0036-1445},
     journal={SIAM Rev.},
      volume={38},
      number={1},
       pages={49\ndash 95},
         url={https://doi-org.ezproxy.lib.purdue.edu/10.1137/1038003},
      review={\MR{1379041}},
}

\bib{Watrous2009}{article}{
      author={Watrous, John},
       title={Semidefinite programs for completely bounded norms},
        date={2009},
     journal={Theory Comput.},
      volume={5},
       pages={217\ndash 238},
  url={https://doi-org.ezproxy.lib.purdue.edu/10.4086/toc.2009.v005a011},
      review={\MR{2592394}},
}

\bib{Watrous2013}{article}{
      author={Watrous, John},
       title={Simpler semidefinite programs for completely bounded norms},
        date={2013},
     journal={Chic. J. Theoret. Comput. Sci.},
       pages={Article 8, 19},
         url={https://doi-org.ezproxy.lib.purdue.edu/10.4086/cjtcs.2013.008},
      review={\MR{3084552}},
}

\end{biblist}
\end{bibdiv}

\end{document}